\documentclass[]{interact}

\usepackage{epstopdf}
\usepackage[caption=false]{subfig}

\usepackage[longnamesfirst,sort]{natbib}
\bibpunct[, ]{(}{)}{;}{a}{,}{,}

\theoremstyle{plain}
\newtheorem{theorem}{Theorem}[section]
\newtheorem{lemma}[theorem]{Lemma}

\theoremstyle{definition}

\theoremstyle{remark}

\begin{document}

\title{Sum of Independent XGamma Distributions}

\author{
\name{Therrar Kadri\textsuperscript{a,b}\thanks{CONTACT T. Kadri. Email: therrar@hotmail.com}, Rahil Omairi\textsuperscript{b}, Khaled Smaili\textsuperscript{c} and Seifedine Kadry \textsuperscript{d}}
\affil{\textsuperscript{a}Department of Science, Northwestern Polytechnic, Grande Praire, Canada;\\ \textsuperscript{b}Department of Art and Science, Lebanese International University, Lebanon\\
\textsuperscript{c}Department of Applied Mathematics, Faculty of Sciences, Lebanese University, Zahle, Lebanon \\
\textsuperscript{d} Department of Applied Data Science, Noroff University College, Norway
}
}

\maketitle

\begin{abstract}
The XGamma distribution is a generated distribution from a mixture of
Exponential and Gamma distributions. It is found that in many cases the XGamma
has more flexibility than the Exponential distribution. In this paper we
consider the sum of independent XGamma distributions with different
parameters. We showed that the probability density function of this
distribution is a sum of the probability density function of the Erlang
distributions. As a consequence, we find exact closed expressions of the other
related statistical functions. Next, we examine the estimation of the
parameters by maximum likelihood estimators. We observe in an applications a
real data set which shows that this model provides better fit to the data as
compared to the sum of the Exponential distributions, the Hypoexponential models.
\end{abstract}


\section{Introduction}

Some random variables play an important role in queuing and survival problems
\citep{1,3}. However, the sum of independent random variables is one
of the most important and used function of random variables that takes a
worthy part in modeling many events in various fields of science
\citep{17}, Markov process, \citep{7} reliability and performance
evaluation \citep{4}, and spatial diversity \citep{9}. The sum of independent
random variables has been studied through many authors over time, as well as
the case of finding their exact probability density function (PDF), cumulative
distribution function (CDF), moment generating function (MGF) and other
statistical parameters. Moreover, we may present some: \citep{2} on the sum of Exponential random variables, \citep{11} on
the sum of Gamma random variables.

The introduction of new lifetime distributions or modified lifetime
distributions has become a time-valuable methods in statistical and biomedical
research. In addition, finite mixture distributions that emerge from standard
distributions plays a better role in modeling real-life phenomena as compared
to the standard ones see \citep{10}. In recent years, it turns out that there
are many well-known distributions used to model data sets do not offer enough
flexibility to provide proper fit. For this reason, new distributions are
introduced that are more flexible and fit the data properly. By maintaining
the role of finite mixture distributions in modeling time-to-event data see
\citep{5}. \citet{12} introduced a new finite mixture of Exponential
and Gamma distributions called the XGamma distribution. They list some of its
statistical parameters and clarify an application that this model is provides
an adequate fit for the data set more than the Exponential distribution. For
this reason, they consider this new finite mixture and found that the XGamma
model provides better fit to the data as compared to the Exponential model.

In this paper, we consider the sum of $n$ independent XGamma random variables
with different scale parameters. We give an exact closed expressions of the
PDF of this distribution as a linear combination of Erlang random variables.
As a consequence, we determine closed expressions of the CDF, MGF and moment
of order k of the convolution of the XGamma distribution and also as a linear
combination of the Erlang random variables. We discuss estimation of the
parameters by maximum likelihood estimators. We observe in an application to
real data set that this model is quite flexible and can be used quite
effectively in analyzing reliability models and indicated that the Mixed
Erlang distribution is a serious competitor to the others.

\subsection{Some Preliminaries}\label{class}

In this section, we state some continuous distributions. Moreover, we give
some of its statistical parameter.

The Erlang distribution is the sum of $n$ independent identical Exponential
random variables with parameter $\theta\in%
\mathbb{R}
_{+}.$ We write it as $E_{\theta}^{n}\thicksim$Erl($n,\theta$). The
probability density function of $E_{\theta}^{n}$ is given as%
\begin{equation}
f_{E_{\theta}^{n}}\left(  t\right)  =\left\{
\begin{array}
[c]{cc}%
\frac{(\theta t)^{n-1}\theta e^{-\theta t}}{(n-1)!} & \text{if }t\geq0\\
0 & \text{if }t<0
\end{array}
\right. \label{PDFErla}%
\end{equation}
and with MGF of $E_{\theta}^{n}$ is%
\begin{equation}
\Phi_{E_{\theta}^{n}}\left(  t\right)  =\frac{\theta^{n}}{(\theta-t)^{n}%
},\text{ for }t>\theta.\label{MGFErlangg}%
\end{equation}
Moreover, the Laplace transform of $\ $\ PDF$\ $of $E_{\theta}^{n}$ is%
\begin{equation}%
\mathcal{L}%
\left\{  f_{E_{\theta}^{n}}(t)\right\}  =\frac{\theta^{n}}{(\theta+t)^{n}%
},\text{ for }t>\theta.\label{LaplacErlang}%
\end{equation}

The Hypoexponential distribution is the distribution of the sum of $m\geq
2$\ independent Exponential random variables with different parameters
presented by \citet{21}. This distribution is used in modeling
multiple exponential stages in series, when those stages are not distinct then
it is considered as a general case of the Hypoexponential distribution. This
general case can be expressed as $H_{\overrightarrow{\theta}}^{\overrightarrow
{k}}=\sum\limits_{i=1}^{n}\sum\limits_{j=1}^{k_{i}}X_{ij}$, where $X_{ij}$ is
an Exponential random variables with parameter $\theta_{i},$ $i=1,2,...,n,$
and written as $H_{\overrightarrow{\theta}}^{\overrightarrow{k}}\thicksim
$Hypoexp($\overrightarrow{\theta},\overrightarrow{k}$), where $\overrightarrow
{\theta}=(\theta_{1},\theta_{2},...,\theta_{n})\in%
\mathbb{R}
_{+}^{n},\ \overrightarrow{k}=(k_{1},k_{2},...,k_{n})\in%
\mathbb{N}
^{n}\ $and $m=\sum\limits_{i=1}^{n}k_{i}.$

The XGamma distribution is a finite mixture of Exponential and Gamma
distributions presented by  \citet{12}. We write it as XG$_{\theta
}\thicksim$XG($\theta$). The PDF of this distribution expressed as
$f_{XG_{\theta}}\left(  t\right)  =\sum\limits_{i=1}^{2}\pi_{i}f_{i}\left(
t\right)  ,$ where $f_{1}\left(  t\right)  $ is the PDF of Exponential
distribution with parameter $\theta$ and $f_{2}(t)$ is the PDF of Gamma
distribution with scale parameter $\theta\ $and shape parameter $3$, and with
mixing proportions $\pi_{1}=\theta/(1+\theta)\ $and $\pi_{2}=1-\pi_{1}.$
Moreover, they expressed the PDF of the XGamma distribution as%
\begin{equation}
f_{XG_{\theta}}(t)=\frac{\theta^{2}}{(1+\theta)}(1+\frac{\theta}{2}%
t^{2})e^{-\theta t},\label{PDFXGamma}%
\end{equation}
with$\ t,\ \theta>0,$ and $\theta\ $is the scale parameter.

Moreover, the MGF\ of XGamma distribution is given as%
\begin{equation}
\Phi_{XG_{\theta}}(t)=\frac{\theta^{2}((\theta-t)^{2}+\theta)}{(1+\theta
)(\theta-t)^{3}}\label{MGFXGamma}%
\end{equation}

\section{The Exact Expression of Sum of Independent XGamma Distributions}

In this section, we find an exact close expressions of the PDF for the sum of
$n$ independent XGamma distribution with different scale parameter $\theta$.
This closed expression is formulated as a linear combination of Erlang distribution.

We suppose that $X_{j}$ are $n$ independent XGamma distributions.

Let $X_{j}\thicksim$XG$(\theta_{j}),$ for $1\leq j\leq n.$ From
Eq(\ref{PDFXGamma}) the PDFs of $X_{j}$ are $f_{X_{j}}(t)=\frac{\theta_{j}%
^{2}}{(1+\theta_{j})}(1+\frac{\theta_{j}}{2}t^{2})e^{-\theta_{j}t}.$

\begin{theorem}
\label{PDF1DiffCase}Let $S_{n}$ be the sum of $n$ independent XGamma
distributions denoted as $S_{n}\thicksim$HypoXG($\overrightarrow{\theta}$). Then the PDF of $S_{n}$ is given as%
\begin{equation}
f_{S_{n}}(t)=%
{\textstyle\sum\limits_{i=1}^{n}}
{\textstyle\sum\limits_{k=1}^{3}}
R_{ik}f_{Y_{ik}}(t)\label{HypoXGPDF2}%
\end{equation}
with%
\begin{equation}
R_{ik}=\frac{A_{ik}}{\theta_{i}^{4-k}}%
{\textstyle\prod\limits_{l=1}^{n}}
\frac{\theta_{l}^{2}}{(1+\theta_{l})},\label{CoeffHypoXG2}%
\end{equation}
and$\ Y_{ik}\thicksim$Erl$\left(  4-k,\theta_{i}\right)  $

Furthermore,%
\begin{align}
A_{i1}  & =\theta_{i}%
{\textstyle\prod\limits_{j=1,j\neq i}^{n}}
\left(  \frac{(\theta_{j}-\theta_{i})^{2}+\theta_{j}}{(\theta_{j}-\theta
_{i})^{3}}\right) \nonumber\\
A_{i2}  & =A_{i1}\left(
{\textstyle\sum\limits_{j=1,j\neq i}^{n}}
\left(  \frac{2(\theta_{j}-\theta_{i})}{(\theta_{j}-\theta_{i})^{2}+\theta
_{j}}-\frac{3}{\theta_{j}-\theta_{i}}\right)  \right) \nonumber\\
A_{i3}  & =\frac{A_{i1}}{2}\left(
{\textstyle\sum\limits_{j=1,j\neq i}^{n}}
\left(  \frac{2(\theta_{j}-\theta_{i})}{(\theta_{j}-\theta_{i})^{2}+\theta
_{j}}-\frac{3}{\theta_{j}-\theta_{i}}\right)  \right)  ^{2}+\nonumber\\
& \frac{A_{i1}}{2}\left(  \frac{2}{\theta_{i}}+%
{\textstyle\sum\limits_{j=1,j\neq i}^{n}}
\left(  -2\frac{\left(  \theta_{j}-\theta_{i}\right)  ^{2}-\theta_{j}}{\left(
(\theta_{j}-\theta_{i})^{2}+\theta_{j}\right)  ^{2}}+\frac{3}{\left(
\theta_{j}-\theta_{i}\right)  ^{2}}\right)  \right)  .\label{eqsAip}%
\end{align}

\end{theorem}

\begin{proof}
Since $X_{i}$ are independent, then $%
\mathcal{L}%
\{f_{S_{n}}(t)\}=\prod\limits_{i=1}^{n}%
\mathcal{L}%
\{f_{X_{i}}(t)\}.$ But using the fact that $%
\mathcal{L}%
\{f_{X_{i}}(t)\}=\Phi_{X_{i}}(-t),$ the MGF of the XGamma is given in
Eq(\ref{MGFXGamma}), then $%
\mathcal{L}%
\{f_{S_{n}}(t)\}=\prod\limits_{i=1}^{n}\frac{\theta_{i}^{2}((\theta_{i}%
+t)^{2}+\theta_{i})}{(1+\theta_{i})(\theta_{i}+t)^{3}}=\left(  \prod
\limits_{i=1}^{n}\frac{\theta_{i}^{2}}{(1+\theta_{i})}\right)
{\textstyle\prod\limits_{i=1}^{n}}
\left(  \frac{(\theta_{i}+t)^{2}+\theta_{i}}{(\theta_{i}+t)^{3}}\right)  .$

Now by using the Heaviside Theorem for Repeated Roots, we have%
\begin{align}%
{\textstyle\prod\limits_{i=1}^{n}}
\left(  \frac{(\theta_{i}+t)^{2}+\theta_{i}}{(\theta_{i}+t)^{3}}\right)   & =%
{\textstyle\sum\limits_{k=1}^{3}}
\frac{A_{1k}}{(\theta_{1}+t)^{4-k}}+%
{\textstyle\sum\limits_{j=1}^{3}}
\frac{A_{2k}}{(\theta_{2}+t)^{4-k}}+...+%
{\textstyle\sum\limits_{p=1}^{3}}
\frac{A_{nk}}{(\theta_{n}+t)^{4-k}}\nonumber\\
& =%
{\textstyle\sum\limits_{i=1}^{n}}
{\textstyle\sum\limits_{k=1}^{3}}
\frac{A_{ik}}{(\theta_{i}+t)^{4-k}}\label{HVSeqs}%
\end{align}
where%
\begin{equation}
A_{ik}=\underset{t\rightarrow-\theta_{i}}{\lim}\frac{1}{(k-1)!}\frac{d^{k-1}%
}{dt^{k-1}}\left(  \left(  t+\theta_{i}\right)  ^{3}%
{\textstyle\prod\limits_{j=1}^{n}}
\left(  \frac{(\theta_{j}+t)^{2}+\theta_{j}}{(\theta_{j}+t)^{3}}\right)
\right) \label{EqAip}%
\end{equation}
for $k=1,2,3$ and $i=1,2,...,n$. Next, we can write $\left(  t+\theta
_{i}\right)  ^{3}%
{\textstyle\prod\limits_{j=1}^{n}}
\left(  \frac{(\theta_{j}+t)^{2}+\theta_{j}}{(\theta_{j}+t)^{3}}\right)
=\left(  (\theta_{i}+t)^{2}+\theta_{i}\right)
{\textstyle\prod\limits_{j=1,j\neq i}^{n}}
\left(  \frac{(\theta_{j}+t)^{2}+\theta_{j}}{(\theta_{j}+t)^{3}}\right)  .$ By
substituting in Eq (\ref{EqAip}), we obtain that%
\[
A_{ik}=\underset{t\rightarrow-\theta_{i}}{\lim}\frac{1}{(k-1)!}\frac{d^{k-1}%
}{dt^{k-1}}\left(  \left(  (\theta_{i}+t)^{2}+\theta_{i}\right)
{\textstyle\prod\limits_{j=1,j\neq i}^{n}}
\left(  \frac{(\theta_{j}+t)^{2}+\theta_{j}}{(\theta_{j}+t)^{3}}\right)
\right)
\]
Next, we derive $A_{ik},$ for $k=1,2,3$ and $i=1,2,...,n$, then we have%
\begin{align}
A_{i1}  & =\underset{s\rightarrow-\theta_{i}}{\lim}\left[  \left(  \left(
s+\theta_{i}\right)  ^{3}%
{\textstyle\prod\limits_{j=1}^{n}}
\left(  \frac{(\theta_{j}+s)^{2}+\theta_{j}}{(\theta_{j}+s)^{3}}\right)
\right)  \right] \nonumber\\
& =\underset{s\rightarrow-\theta_{i}}{\lim}\left[  \left(  (\theta_{i}%
+s)^{2}+\theta_{i}\right)
{\textstyle\prod\limits_{j=1,j\neq i}^{n}}
\left(  \frac{(\theta_{j}+s)^{2}+\theta_{j}}{(\theta_{j}+s)^{3}}\right)
\right] \nonumber\\
& =\theta_{i}%
{\textstyle\prod\limits_{j=1,j\neq i}^{n}}
\left(  \frac{(\theta_{j}-\theta_{i})^{2}+\theta_{j}}{(\theta_{j}-\theta
_{i})^{3}}\right)  .\label{EqAi11}%
\end{align}
On the other hand,%
\begin{equation}
A_{i2}=\underset{s\rightarrow-\theta_{i}}{\lim}\left[  \left(  \left(
s+\theta_{i}\right)  ^{3}%
{\textstyle\prod\limits_{j=1}^{n}}
\left(  \frac{(\theta_{j}+s)^{2}+\theta_{j}}{(\theta_{j}+s)^{3}}\right)
\right)  ^{\prime}\right] \label{EqAi2}%
\end{equation}
Let $Y=\left(  \left(  s+\theta_{i}\right)  ^{3}%
{\textstyle\prod\limits_{j=1}^{n}}
\left(  \frac{(\theta_{j}+s)^{2}+\theta_{j}}{(\theta_{j}+s)^{3}}\right)
\right)  ,$

then $\ln Y=3\ln\left(  s+\theta_{i}\right)  +%
{\textstyle\sum\limits_{j=1}^{n}}
\ln\left(  (\theta_{j}+s)^{2}+\theta_{j}\right)  -3\ln(\theta_{j}+s)),$

which implies $\frac{Y^{\prime}}{Y}=\frac{3}{s+\theta_{i}}+%
{\textstyle\sum\limits_{j=1}^{n}}
\left(  \frac{2(\theta_{j}+s)}{(\theta_{j}+s)^{2}+\theta_{j}}-\frac
{3}{s+\theta_{j}}\right)  ,$

thus$\ Y^{\prime}=Y\left(  \frac{3}{s+\theta_{i}}+%
{\textstyle\sum\limits_{j=1}^{n}}
\left(  \frac{2(\theta_{j}+s)}{(\theta_{j}+s)^{2}+\theta_{j}}-\frac
{3}{s+\theta_{j}}\right)  \right)  .$ By substituting in Eq (\ref{EqAi2}), we
obtain%
\begin{align}
A_{i2}  & =\underset{s\rightarrow-\theta_{i}}{\lim}\left[  Y^{\prime}\right]
\nonumber\\
& =A_{i1}%
{\textstyle\sum\limits_{j=1,j\neq i}^{n}}
\left(  \frac{2(\theta_{j}-\theta_{i})}{(\theta_{j}-\theta_{i})^{2}+\theta
_{j}}-\frac{3}{\theta_{j}-\theta_{i}}\right)  .\label{EqAi22}%
\end{align}

Eventually,%
\begin{equation}
A_{i3}=\underset{s\rightarrow-\theta_{i}}{\lim}\frac{1}{2}\left(
Y^{^{\prime\prime}}\right) \label{EqAi3}%
\end{equation}

In the same manner, we have

$Y^{\prime\prime}=Y^{\prime}\left(  \frac{3}{s+\theta_{i}}+%
{\textstyle\sum\limits_{j=1}^{n}}
\left(  \frac{2(\theta_{j}+s)}{(\theta_{j}+s)^{2}+\theta_{j}}-\frac
{3}{s+\theta_{j}}\right)  \right)  $

$+\left(  -\frac{3}{\left(  s+\theta_{i}\right)  ^{2}}+%
{\textstyle\sum\limits_{j=1}^{n}}
\left(  -2\frac{\left(  s+\theta_{j}\right)  ^{2}-\theta_{j}}{\left(
(\theta_{j}+s)^{2}+\theta_{j}\right)  ^{2}}+\frac{3}{\left(  s+\theta
_{j}\right)  ^{2}}\right)  \right)  Y.$ Again by substituting in Eq
(\ref{EqAi3}), we get%
\begin{align}
A_{i,3}  & =\underset{s\rightarrow-\theta_{i}}{\lim}\frac{1}{2}\left[
Y^{^{\prime\prime}}\right] \nonumber\\
& =\frac{1}{2}A_{i2}\left(
{\textstyle\sum\limits_{j=1,j\neq i}^{n}}
\left(  \frac{2(\theta_{j}-\theta_{i})}{(\theta_{j}-\theta_{i})^{2}+\theta
_{j}}-\frac{3}{\theta_{j}-\theta_{i}}\right)  \right) \nonumber\\
& +\frac{1}{2}A_{i1}\left(  \frac{2}{\theta_{i}}+%
{\textstyle\sum\limits_{j=1,j\neq i}^{n}}
\left(  -2\frac{\left(  \theta_{j}-\theta_{i}\right)  ^{2}-\theta_{j}}{\left(
(\theta_{j}-\theta_{i})^{2}+\theta_{j}\right)  ^{2}}+\frac{3}{\left(
\theta_{j}-\theta_{i}\right)  ^{2}}\right)  \right) \nonumber\\
& =\frac{1}{2}A_{i1}\left(
{\textstyle\sum\limits_{j=1,j\neq i}^{n}}
\left(  \frac{2(\theta_{j}-\theta_{i})}{(\theta_{j}-\theta_{i})^{2}+\theta
_{j}}-\frac{3}{\theta_{j}-\theta_{i}}\right)  \right)  ^{2}\nonumber\\
& +\frac{1}{2}A_{i1}\left(  \frac{2}{\theta_{i}}+%
{\textstyle\sum\limits_{j=1,j\neq i}^{n}}
\left(  -2\frac{\left(  \theta_{j}-\theta_{i}\right)  ^{2}-\theta_{j}}{\left(
(\theta_{j}-\theta_{i})^{2}+\theta_{j}\right)  ^{2}}+\frac{3}{\left(
\theta_{j}-\theta_{i}\right)  ^{2}}\right)  \right) \label{EqAi33}%
\end{align}

However, by substituting Eq(\ref{HVSeqs}) in the above Laplacian
transformation, we obtain%
\begin{equation}%
\mathcal{L}%
\left\{  f_{S_{n}}(t)\right\}  =\left(  \prod\limits_{i=1}^{n}\frac{\theta
_{i}^{2}}{(1+\theta_{i})}\right)
{\textstyle\sum\limits_{i=1}^{n}}
{\textstyle\sum\limits_{k=1}^{3}}
\frac{A_{ik}}{(\theta_{i}+t)^{4-k}}\label{Eq4}%
\end{equation}
But from Eq(\ref{LaplacErlang}), we have $%
\mathcal{L}%
\left\{  f_{E_{\theta}^{n}}(t)\right\}  =\frac{\theta^{n}}{(\theta+t)^{n}},$
and thus we rewrite Eq(\ref{Eq4}) as%
\begin{align*}%
\mathcal{L}%
\left\{  f_{S_{n}}(t)\right\}   & =\left(  \prod\limits_{i=1}^{n}\frac
{\theta_{i}^{2}}{(1+\theta_{i})}\right)
{\textstyle\sum\limits_{i=1}^{n}}
{\textstyle\sum\limits_{k=1}^{3}}
\frac{A_{ik}}{\theta_{i}^{4-k}}\frac{\theta_{i}^{4-k}}{(\theta_{i}+s)^{4-k}}\\
& =\left(  \prod\limits_{i=1}^{n}\frac{\theta_{i}^{2}}{(1+\theta_{i})}\right)
%
{\textstyle\sum\limits_{i=1}^{n}}
{\textstyle\sum\limits_{k=1}^{3}}
\frac{A_{ik}}{\theta_{i}^{4-k}}%
\mathcal{L}%
\left\{  f_{E_{\theta_{i}}^{4-k}}(t)\right\} \\
& =\left(  \prod\limits_{i=1}^{n}\frac{\theta_{i}^{2}}{(1+\theta_{i})}\right)
%
{\textstyle\sum\limits_{i=1}^{n}}
{\textstyle\sum\limits_{k=1}^{3}}
\frac{A_{ik}}{\theta_{i}^{4-k}}f_{Y_{ik}}(t)\\
& =%
{\textstyle\sum\limits_{i=1}^{n}}
{\textstyle\sum\limits_{k=1}^{3}}
R_{ik}f_{Y_{ik}}(t)
\end{align*}
where $R_{ik}=\left(  \prod\limits_{i=1}^{n}\frac{\theta_{i}^{2}}%
{(1+\theta_{i})}\right)  \frac{A_{ik}}{\theta_{i}^{4-k}},$ and $A_{ik}$ are
defined in Eqs (\ref{EqAi11},\ref{EqAi22},\ref{EqAi33}) respectively, and
$Y_{ik}\thicksim$Erl$\left(  4-k,\theta_{i}\right)  .$
\end{proof}

In the next theorem, we give the CDF, MGF, reliability function and hazard
function expressions of $S_{n}$.

\begin{theorem}
\label{cdfDiffcase}Let $S_{n}\thicksim$HypoXG$(\overrightarrow{\theta})$. Then we have%
\begin{equation}
F_{S_{n}}(t)=%
{\textstyle\sum\limits_{i=1}^{n}}
{\textstyle\sum\limits_{k=1}^{3}}
R_{ik}F_{Y_{ik}}(t)\label{Eq_CDFdiffcase2}%
\end{equation}

\end{theorem}

where $R_{ik}$ is defined in Eq (\ref{CoeffHypoXG2}), and $Y_{ik}\thicksim
$Erl$\left(  4-k,\theta_{i}\right)  .$

\begin{proof}
From Theorem $\ref{PDF1DiffCase}$ the PDF of $S_{n}$ is $f_{S_{n}}(t)=%
{\textstyle\sum\limits_{i=1}^{n}}
{\textstyle\sum\limits_{k=1}^{3}}
R_{ik}f_{Y_{ik}}(t).$ On the other hand we have the CDF of $S_{n}$ can be
expressed as%
\[
F_{S_{n}}(t)=%
{\displaystyle\int\limits_{0}^{t}}
f_{S_{n}}(x)dx=%
{\displaystyle\int\limits_{0}^{t}}
{\textstyle\sum\limits_{i=1}^{n}}
{\textstyle\sum\limits_{k=1}^{3}}
R_{ik}f_{Y_{ik}}(x)dx=%
{\textstyle\sum\limits_{i=1}^{n}}
{\textstyle\sum\limits_{k=1}^{3}}
R_{ik}%
{\displaystyle\int\limits_{0}^{t}}
f_{Y_{ik}}(x)dx=%
{\textstyle\sum\limits_{i=1}^{n}}
{\textstyle\sum\limits_{k=1}^{3}}
R_{ik}F_{Y_{ik}}(t)
\]
where, $Y_{ik}\thicksim$Erl$\left(  4-k,\theta_{i}\right)  .$
\end{proof}

\begin{lemma}
\label{coeffDiffcaseis1}Let $R_{ik}$ be the coefficient of HypoXG, then $%
{\textstyle\sum\limits_{i=1}^{n}}
{\textstyle\sum\limits_{k=1}^{3}}
R_{ik}=1$.
\end{lemma}

\begin{proof}
Let $F_{Y_{ik}}(t)$ and $F_{S_{n}}(t)$ be the CDF of $Y_{ik}\thicksim
$Erl$\left(  4-k,\theta_{i}\right)  $ and $S_{n}\thicksim$%
HypoXG$(\overrightarrow{\theta})$ respectively. However the
CDF for any random variable $X$ is $P\left(  X<t\right)  =1$ as $t\rightarrow
+\infty,$ thus $\underset{t\rightarrow+\infty}{\lim}F_{Y_{ik}}(t)=1$ and
$\underset{t\rightarrow+\infty}{\lim}F_{S_{n}}(t)=1,$ but form Theorem
\ref{cdfDiffcase}$\ $we have $F_{S_{n}}(t)=%
{\textstyle\sum\limits_{i=1}^{n}}
{\textstyle\sum\limits_{k=1}^{3}}
R_{ik}F_{Y_{ik}}(t),$ then $\underset{t\rightarrow+\infty}{\lim}F_{S_{n}%
}(t)=\underset{t\rightarrow+\infty}{\lim}%
{\textstyle\sum\limits_{i=1}^{n}}
{\textstyle\sum\limits_{k=1}^{3}}
R_{ik}F_{Y_{ik}}(t)=%
{\textstyle\sum\limits_{i=1}^{n}}
{\textstyle\sum\limits_{k=1}^{3}}
R_{ik}\underset{t\rightarrow+\infty}{\lim}F_{Y_{ik}}(t).$ Hence, $%
{\textstyle\sum\limits_{i=1}^{n}}
{\textstyle\sum\limits_{k=1}^{3}}
R_{ik}=1.$
\end{proof}

\begin{theorem}
\label{MGFdiffcase}Let $S_{n}\thicksim$HypoXG$(\overrightarrow{\theta
})$. Then we have%
\begin{equation}
\Phi_{S_{n}}(t)=%
{\textstyle\sum\limits_{i=1}^{n}}
{\textstyle\sum\limits_{k=1}^{3}}
R_{ik}\Phi_{Y_{ik}}(t)\label{Eq_MGFdiffcase2}%
\end{equation}

\end{theorem}

\begin{proof}
Referring to Theorem $\ref{PDF1DiffCase}$\ the PDF of $S_{n}$ is $f_{S_{n}%
}(t)=%
{\textstyle\sum\limits_{i=1}^{n}}
{\textstyle\sum\limits_{k=1}^{3}}
R_{ik}f_{Y_{ik}}(t).\ $However, using the MGF expression, we obtain%
\[
\Phi_{S_{n}}(t)=\int\limits_{-\infty}^{+\infty}e^{tx}f_{S_{n}}(x)dx=\int
\limits_{-\infty}^{+\infty}e^{tx}\left(
{\textstyle\sum\limits_{i=1}^{n}}
{\textstyle\sum\limits_{k=1}^{3}}
R_{ik}f_{Y_{ik}}(x)\right)  dx=%
{\textstyle\sum\limits_{i=1}^{n}}
{\textstyle\sum\limits_{k=1}^{3}}
R_{ik}\int\limits_{-\infty}^{+\infty}e^{tx}f_{Y_{ik}}(x)dx
\]
but%
\[
\int\limits_{-\infty}^{+\infty}e^{tx}f_{Y_{ik}}(x)dx=\Phi_{Y_{ik}}\left(
t\right)  ,
\]
thus%
\[
\Phi_{S_{n}}(t)=%
{\textstyle\sum\limits_{i=1}^{n}}
{\textstyle\sum\limits_{k=1}^{3}}
R_{ik}\Phi_{Y_{ik}}(t).
\]
where $\Phi_{Y_{ik}}(t)$ is defined in Eq (\ref{MGFErlangg}).
\end{proof}

\begin{theorem}
Let $S_{n}\thicksim$HypoXG$(\overrightarrow{\theta})$. Then
we have%
\begin{equation}
R_{S_{n}}(t)=%
{\textstyle\sum\limits_{i=1}^{n}}
{\textstyle\sum\limits_{k=1}^{3}}
R_{ik}R_{Y_{ik}}(t)\label{Eq_Reldiffcase2}%
\end{equation}

\end{theorem}

\begin{proof}
Using the expression of reliability function, we have%
\[
R_{S_{n}}\left(  t\right)  =1-F_{S_{n}}(t)=1-%
{\textstyle\sum\limits_{i=1}^{n}}
{\textstyle\sum\limits_{k=1}^{3}}
R_{ik}F_{Y_{ik}}(t),
\]
but%
\[
F_{Y_{ik}}(t)=1-R_{Y_{ik}}(t),
\]
then,%
\begin{align*}
R_{S_{n}}\left(  t\right)   & =1-%
{\textstyle\sum\limits_{i=1}^{n}}
{\textstyle\sum\limits_{k=1}^{3}}
R_{ik}\left(  1-R_{Y_{ik}}(t)\right) \\
& =1-%
{\textstyle\sum\limits_{i=1}^{n}}
{\textstyle\sum\limits_{k=1}^{3}}
R_{ik}+%
{\textstyle\sum\limits_{i=1}^{n}}
{\textstyle\sum\limits_{k=1}^{3}}
R_{ik}R_{Y_{ik}}(t)
\end{align*}
but by using Lemma \ref{coeffDiffcaseis1} we have $%
{\textstyle\sum\limits_{i=1}^{n}}
{\textstyle\sum\limits_{k=1}^{3}}
R_{ik}=1,$ hence $R_{S_{n}}(t)=\sum\limits_{i=0}^{n}A_{i}R_{Y_{i}}(t).$
\end{proof}

\begin{theorem}
Let $S_{n}\thicksim$HypoXG$(\overrightarrow{\theta})$. Then
we have%
\begin{equation}
h_{S_{n}}(t)=\frac{%
{\textstyle\sum\limits_{i=1}^{n}}
{\textstyle\sum\limits_{k=1}^{3}}
R_{ik}h_{Y_{ik}}\left(  t\right)  R_{Y_{ik}}(t)}{%
{\textstyle\sum\limits_{i=1}^{n}}
{\textstyle\sum\limits_{k=1}^{3}}
R_{ik}R_{Y_{ik}}(t)}\label{Eq_hazarddiffcase2}%
\end{equation}

\end{theorem}

\begin{proof}
Referring to the expression of the hazard function, we have%
\[
h_{S_{n}}(t)=\frac{f_{S_{n}}(t)}{R_{S_{n}}(t)}=\frac{%
{\textstyle\sum\limits_{i=1}^{n}}
{\textstyle\sum\limits_{k=1}^{3}}
R_{ik}f_{Y_{ik}}(t)}{%
{\textstyle\sum\limits_{i=1}^{n}}
{\textstyle\sum\limits_{k=1}^{3}}
R_{ik}R_{Y_{ik}}(t)}%
\]
but $f_{Y_{ik}}\left(  t\right)  =h_{Y_{ik}}\left(  t\right)  R_{Y_{ik}%
}\left(  t\right)  ,$ thus%
\[
h_{S_{n}}(t)=\frac{%
{\textstyle\sum\limits_{i=1}^{n}}
{\textstyle\sum\limits_{k=1}^{3}}
R_{ik}h_{Y_{ik}}\left(  t\right)  R_{Y_{ik}}(t)}{%
{\textstyle\sum\limits_{i=1}^{n}}
{\textstyle\sum\limits_{k=1}^{3}}
R_{ik}R_{Y_{ik}}(t)}.
\]

\end{proof}

\begin{theorem}
Let $S_{n}\thicksim$HypoXG$(\overrightarrow{\theta})$. Then
we have%
\begin{equation}
E[S_{n}^{k}]=%
{\textstyle\sum\limits_{i=1}^{n}}
{\textstyle\sum\limits_{k=1}^{3}}
R_{ik}E[Y_{ik}^{k}]\label{Eq_ExpKdiffcase2}%
\end{equation}

\end{theorem}

\begin{proof}
From Theorem \ref{MGFdiffcase}, we have $\Phi_{S_{n}}(t)=%
{\textstyle\sum\limits_{i=1}^{n}}
{\textstyle\sum\limits_{k=1}^{3}}
R_{ik}\Phi_{Y_{ik}}(t).$ But the moment of order $k$ of $S_{n}$ is expressed
as $E[S_{n}^{k}]=\left.  \frac{d^{k}\Phi_{S_{n}}(t)}{dt^{k}}\right\vert
_{t=0}$ then we obtain $E[S_{n}^{k}]=\left.  \frac{d^{k}\Phi_{S_{n}}%
(t)}{dt^{k}}\right\vert _{t=0}=%
{\textstyle\sum\limits_{i=1}^{n}}
{\textstyle\sum\limits_{k=1}^{3}}
R_{ik}\left.  \frac{d^{k}\Phi_{Y_{ik}}(t)}{dt^{k}}\right\vert _{t=0}=%
{\textstyle\sum\limits_{i=1}^{n}}
{\textstyle\sum\limits_{k=1}^{3}}
R_{ik}E[Y_{ik}^{k}].$
\end{proof}

\section{Maximum Likelihood Estimation of Parameters in HypoXGamma
Distribution}

In this section, we examine the estimation of the parameters of the HypoXGamma
by maximum likelihood estimators. Next, we observe in an applications a real
data set which shows that this model provides better fit to the data as
compared to the sum of the Exponential distributions, the Hypoexponential models

\begin{theorem}
Let $S_{n}\sim$HypoXG($\overrightarrow{\theta}$)$,$ $\overrightarrow{\theta
}=(\theta_{1},\theta_{2},...,\theta_{n})$ and $t_{1},t_{2},...,t_{N}$ be the
observed values. Then the maximum likelihood estimators of $\theta_{p}$, for
$p=1,2,...,n$. is denoted by $\widehat{\theta}_{p}$ that verify the following
implicit equations%
\begin{equation}%
{\textstyle\sum\limits_{u=1}^{N}}
\frac{\frac{N\left(  2+\theta_{p}\right)  }{\theta_{p}\left(  1+\theta
_{p}\right)  }+%
{\textstyle\sum\limits_{u=1}^{N}}
\frac{%
{\textstyle\sum\limits_{i=1}^{n}}
\frac{\partial}{\partial\theta_{p}}\left(  e^{-\theta_{i}t_{u}}\left(
\frac{A_{i,1}t_{u}^{2}}{2}+A_{i,2}t_{u}+A_{i,3}\right)  \right)  }{%
{\textstyle\sum\limits_{i=1}^{n}}
e^{-\theta_{i}t_{u}}\left(  \frac{A_{i,1}t_{u}^{2}}{2}+A_{i,2}t_{u}%
+A_{i,3}\right)  }}{%
{\textstyle\sum\limits_{u=1}^{N}}
\log K+%
{\textstyle\sum\limits_{u=1}^{N}}
\log\left[
{\textstyle\sum\limits_{i=1}^{n}}
\left(  \frac{A_{i,1}}{\theta_{i}^{3}}\frac{\theta_{i}^{3}t_{u}^{2}%
e^{-\theta_{i}t_{u}}}{2}+\frac{A_{i,2}}{\theta_{i}^{2}}\theta_{i}^{2}%
t_{u}e^{-\theta_{i}t_{u}}+\frac{A_{i,3}}{\theta_{i}}\theta_{i}e^{-\theta
_{i}t_{u}}\right)  \right]  }=0\label{EQ_impliciteFct2}%
\end{equation}
where$\ K=%
{\textstyle\prod\limits_{i=1}^{n}}
\frac{\theta_{i}^{2}}{(1+\theta_{i})},$ and $A_{i,1},A_{i,2}$ and $A_{i,3}$
defined in Theorem \ref{PDF1DiffCase}.
\end{theorem}

\begin{proof}
Since $S_{n}\sim$HypoXG($\overrightarrow{\theta}$), with $\overrightarrow
{\theta}=(\theta_{1},\theta_{2},...,\theta_{n})$ and $t_{1},t_{2},...,t_{N}$
the observed values. Then from Theorem \ref{PDF1DiffCase}, the probability
density function of $S_{n}$ is given as%
\begin{align*}
f_{S_{n}}(t)  & =%
{\textstyle\sum\limits_{i=1}^{n}}
{\textstyle\sum\limits_{k=1}^{3}}
R_{ik}f_{Y_{ik}}(t)\\
& =%
{\textstyle\sum\limits_{i=1}^{n}}
R_{i1}\frac{\theta_{i}^{3}t^{2}e^{-\theta_{i}t}}{2}+R_{i2}\theta_{i}%
^{2}te^{-\theta_{i}t}+R_{i3}\theta_{i}e^{-\theta_{i}t}%
\end{align*}
Now the likelihood function is given by%
\[
L(\overrightarrow{\theta})=\prod\limits_{u=1}^{N}%
{\textstyle\sum\limits_{i=1}^{n}}
R_{i1}\frac{\theta_{i}^{3}t_{u}^{2}e^{-\theta_{i}t_{u}}}{2}+R_{i2}\theta
_{i}^{2}t_{u}e^{-\theta_{i}t_{u}}+R_{i3}\theta_{i}e^{-\theta_{i}t_{u}}%
\]
with%
\[
R_{i1}=K\frac{A_{i,1}}{\theta_{i}^{3}},\ R_{i2}=K\frac{A_{i,2}}{\theta_{i}%
^{2}},\ R_{i3}=K\frac{A_{i,3}}{\theta_{i}}%
\]
where$\ K=%
{\textstyle\prod\limits_{i=1}^{n}}
\frac{\theta_{i}^{2}}{(1+\theta_{i})},$ and $A_{ik}$ defined in Theorem
\ref{PDF1DiffCase}, then%
\[
L(\overrightarrow{\theta})=\prod\limits_{u=1}^{N}K%
{\textstyle\sum\limits_{i=1}^{n}}
\left(  \frac{A_{i,1}}{\theta_{i}^{3}}\frac{\theta_{i}^{3}t_{u}^{2}%
e^{-\theta_{i}t_{u}}}{2}+\frac{A_{i,2}}{\theta_{i}^{2}}\theta_{i}^{2}%
t_{u}e^{-\theta_{i}t_{u}}+\frac{A_{i,3}}{\theta_{i}}\theta_{i}e^{-\theta
_{i}t_{u}}\right)
\]
thus the log-likelihood function is given by%
\begin{align*}
l(\overrightarrow{\theta})  & =\log\left(  L(\overrightarrow{\theta})\right)
=%
{\textstyle\sum\limits_{u=1}^{N}}
\log\left[  K%
{\textstyle\sum\limits_{i=1}^{n}}
\left(  \frac{A_{i,1}}{\theta_{i}^{3}}\frac{\theta_{i}^{3}t_{u}^{2}%
e^{-\theta_{i}t_{u}}}{2}+\frac{A_{i,2}}{\theta_{i}^{2}}\theta_{i}^{2}%
t_{u}e^{-\theta_{i}t_{u}}+\frac{A_{i,3}}{\theta_{i}}\theta_{i}e^{-\theta
_{i}t_{u}}\right)  \right]  \\
& =%
{\textstyle\sum\limits_{u=1}^{N}}
\log K+%
{\textstyle\sum\limits_{u=1}^{N}}
\log\left[
{\textstyle\sum\limits_{i=1}^{n}}
\left(  \frac{A_{i,1}}{\theta_{i}^{3}}\frac{\theta_{i}^{3}t_{u}^{2}%
e^{-\theta_{i}t_{u}}}{2}+\frac{A_{i,2}}{\theta_{i}^{2}}\theta_{i}^{2}%
t_{u}e^{-\theta_{i}t_{u}}+\frac{A_{i,3}}{\theta_{i}}\theta_{i}e^{-\theta
_{i}t_{u}}\right)  \right]  .
\end{align*}
The nonlinear log-likelihood equations are listed below%
\begin{equation}
\frac{\partial l(\overrightarrow{\theta})}{\partial\theta_{p}}=%
{\textstyle\sum\limits_{u=1}^{N}}
\frac{\frac{\partial}{\partial\theta_{p}}\left(
{\textstyle\sum\limits_{u=1}^{N}}
\log K+%
{\textstyle\sum\limits_{u=1}^{N}}
\log\left[
{\textstyle\sum\limits_{i=1}^{n}}
\left(  \frac{A_{i,1}}{\theta_{i}^{3}}\frac{\theta_{i}^{3}t_{u}^{2}%
e^{-\theta_{i}t_{u}}}{2}+\frac{A_{i,2}}{\theta_{i}^{2}}\theta_{i}^{2}%
t_{u}e^{-\theta_{i}t_{u}}+\frac{A_{i,3}}{\theta_{i}}\theta_{i}e^{-\theta
_{i}t_{u}}\right)  \right]  \right)  }{%
{\textstyle\sum\limits_{u=1}^{N}}
\log K+%
{\textstyle\sum\limits_{u=1}^{N}}
\log\left[
{\textstyle\sum\limits_{i=1}^{n}}
\left(  \frac{A_{i,1}}{\theta_{i}^{3}}\frac{\theta_{i}^{3}t_{u}^{2}%
e^{-\theta_{i}t_{u}}}{2}+\frac{A_{i,2}}{\theta_{i}^{2}}\theta_{i}^{2}%
t_{u}e^{-\theta_{i}t_{u}}+\frac{A_{i,3}}{\theta_{i}}\theta_{i}e^{-\theta
_{i}t_{u}}\right)  \right]  }=0\label{eq_maxdiffcase2}%
\end{equation}
for $p=1,2,...,n$.

Now, we investigate the numerator. We have%
\begin{align*}%
{\textstyle\sum\limits_{u=1}^{N}}
\log K  & =%
{\textstyle\sum\limits_{u=1}^{N}}
\log\left[
{\textstyle\prod\limits_{i=1}^{n}}
\frac{\theta_{i}^{2}}{(1+\theta_{i})}\right] \\
& =N\left(  2\log\left(  \theta_{p}\right)  -\log\left(  1+\theta_{p}\right)
\right)  +N\underset{i\neq p}{%
{\textstyle\sum\limits_{i=1}^{n}}
}\left(  2\log\left(  \theta_{i}\right)  -\log\left(  1+\theta_{i}\right)
\right)
\end{align*}
then%
\[
\frac{\partial}{\partial\theta_{p}}\left(
{\textstyle\sum\limits_{u=1}^{N}}
\log K\right)  =\frac{N\left(  2+\theta_{p}\right)  }{\theta_{p}\left(
1+\theta_{p}\right)  }.
\]
Moreover,%
\begin{align*}
& \frac{\partial}{\partial\theta_{p}}\left(
{\textstyle\sum\limits_{u=1}^{N}}
\log\left[
{\textstyle\sum\limits_{i=1}^{n}}
e^{-\theta_{i}t_{u}}\left(  \frac{A_{i,1}t_{u}^{2}}{2}+A_{i,2}t_{u}%
+A_{i,3}\right)  \right]  \right) \\
& =%
{\textstyle\sum\limits_{u=1}^{N}}
\frac{%
{\textstyle\sum\limits_{i=1}^{n}}
\frac{\partial}{\partial\theta_{p}}\left(  e^{-\theta_{i}t_{u}}\left(
\frac{A_{i,1}t_{u}^{2}}{2}+A_{i,2}t_{u}+A_{i,3}\right)  \right)  }{%
{\textstyle\sum\limits_{i=1}^{n}}
e^{-\theta_{i}t_{u}}\left(  \frac{A_{i,1}t_{u}^{2}}{2}+A_{i,2}t_{u}%
+A_{i,3}\right)  }%
\end{align*}
Again, for the numerator we have%
\begin{align}
& \frac{\partial}{\partial\theta_{p}}\left(  e^{-\theta_{i}t_{u}}\left(
\frac{A_{i,1}t_{u}^{2}}{2}+A_{i,2}t_{u}+A_{i,3}\right)  \right) \nonumber\\
& =\left\{
\begin{tabular}
[c]{ll}%
$\frac{t_{u}^{2}}{2}\frac{\partial\left(  A_{p,1}e^{-\theta_{p}t_{u}}\right)
}{\partial\theta_{p}}+t_{u}\frac{\partial\left(  A_{p,2}e^{-\theta_{p}t_{u}%
}\right)  }{\partial\theta_{p}}+\frac{\partial\left(  A_{p,3}e^{-\theta
_{p}t_{u}}\right)  }{\partial\theta_{p}}$ & if $i=p $\\
$e^{-\theta_{i}t_{u}}\left(  \frac{t_{u}^{2}}{2}\frac{\partial A_{i,1}%
}{\partial\theta_{p}}+t_{u}\frac{\partial A_{i,2}}{\partial\theta_{p}}%
+\frac{\partial A_{i,3}}{\partial\theta_{p}}\right)  $ & if $i\neq p$%
\end{tabular}
\right. \label{1}%
\end{align}
then, if $i=p,$ we have

$A_{p,1}=%
{\textstyle\prod\limits_{j=1,j\neq p}^{n}}
\left(  \frac{(\theta_{j}-\theta_{p})^{2}+\theta_{j}}{(\theta_{j}-\theta
_{p})^{3}}\right)  ,$ so let $\ln A_{p,1}=%
{\textstyle\sum\limits_{j=1,j\neq p}}
\left(  \ln\left(  (\theta_{j}-\theta_{p})^{2}+\theta_{j}\right)  -3\ln
(\theta_{j}-\theta_{p})\right)  ,$

thus $\frac{\frac{\partial A_{p,1}}{\partial\theta_{p}}}{A_{p,1}}=%
{\textstyle\sum\limits_{j=1,j\neq p}^{n}}
\frac{(\theta_{j}-\theta_{p})^{2}+3\theta_{j}}{\left(  (\theta_{j}-\theta
_{p})^{2}+\theta_{j}\right)  \left(  \theta_{j}-\theta_{p}\right)  }$ and
$\frac{\partial A_{p,1}}{\partial\theta_{p}}=A_{p,1}%
{\textstyle\sum\limits_{j=1,j\neq p}^{n}}
\frac{(\theta_{j}-\theta_{p})^{2}+3\theta_{j}}{\left(  (\theta_{j}-\theta
_{p})^{2}+\theta_{j}\right)  \left(  \theta_{j}-\theta_{p}\right)  },$ thus%
\[
\frac{\partial A_{p,1}e_{u}^{-\theta_{p}t_{u}}}{\partial\theta_{p}}%
=A_{p,1}e_{u}^{-\theta_{p}t_{u}}\left(  -t_{u}+%
{\textstyle\sum\limits_{j=1,j\neq p}^{n}}
\frac{(\theta_{j}-\theta_{p})^{2}+3\theta_{j}}{\left(  (\theta_{j}-\theta
_{p})^{2}+\theta_{j}\right)  \left(  \theta_{j}-\theta_{p}\right)  }\right)  .
\]
In the same manner, we get%
\[
\frac{\partial A_{p,2}e_{u}^{-\theta_{p}t_{u}}}{\partial\theta_{p}}%
=A_{p,2}e_{u}^{-\theta_{p}t_{u}}\left(  -t_{u}+V\left(  \theta_{p}\right)
\right)
\]
\qquad where $V\left(  \theta_{p}\right)  =%
{\textstyle\sum\limits_{j=1,j\neq p}^{n}}
\frac{(\theta_{j}-\theta_{p})^{2}+3\theta_{j}}{\left(  (\theta_{j}-\theta
_{p})^{2}+\theta_{j}\right)  \left(  \theta_{j}-\theta_{p}\right)  }+\frac{%
{\textstyle\sum\limits_{j=1,j\neq p}^{n}}
\left(  \frac{4(\theta_{j}-\theta_{p})^{2}}{\left(  (\theta_{j}-\theta
_{p})^{2}+\theta_{j}\right)  ^{2}}-\frac{2}{\left(  \theta_{j}-\theta
_{p}\right)  ^{2}+\theta_{j}}-\frac{3}{\left(  \theta_{j}-\theta_{p}\right)
^{2}}\right)  }{%
{\textstyle\sum\limits_{j=1,j\neq p}^{n}}
\left(  \frac{2(\theta_{j}-\theta_{p})}{(\theta_{j}-\theta_{p})^{2}+\theta
_{j}}-\frac{3}{\theta_{j}-\theta_{p}}\right)  }$

and%
\[
\frac{\partial A_{p,3}e_{u}^{-\theta_{p}t_{u}}}{\partial\theta_{p}}%
=A_{p,3}e_{u}^{-\theta_{p}t_{u}}\left(  -t_{u}+G\left(  \theta_{p}\right)
\right)
\]

where $G\left(  \theta_{p}\right)  =%
{\textstyle\sum\limits_{j=1,j\neq p}^{n}}
\frac{(\theta_{j}-\theta_{p})^{2}+3\theta_{j}}{\left(  (\theta_{j}-\theta
_{p})^{2}+\theta_{j}\right)  \left(  \theta_{j}-\theta_{p}\right)  }$

$+\frac{2\left(
{\textstyle\sum\limits_{j=1,j\neq p}^{n}}
\left(  \frac{4(\theta_{j}-\theta_{p})^{2}}{\left(  (\theta_{j}-\theta
_{p})^{2}+\theta_{j}\right)  ^{2}}-\frac{2}{\left(  \theta_{j}-\theta
_{p}\right)  ^{2}+\theta_{j}}-\frac{3}{\left(  \theta_{j}-\theta_{p}\right)
^{2}}\right)  \right)  \left(
{\textstyle\sum\limits_{j=1,j\neq p}^{n}}
\left(  \frac{2(\theta_{j}-\theta_{p})}{(\theta_{j}-\theta_{p})^{2}+\theta
_{j}}-\frac{3}{\theta_{j}-\theta_{p}}\right)  \right)  }{\left(
{\textstyle\sum\limits_{j=1,j\neq p}^{n}}
\left(  \frac{2(\theta_{j}-\theta_{p})}{(\theta_{j}-\theta_{p})^{2}+\theta
_{j}}-\frac{3}{\theta_{j}-\theta_{p}}\right)  \right)  ^{2}+\left(  \frac
{2}{\theta_{p}}+%
{\textstyle\sum\limits_{j=1,j\neq p}^{n}}
\left(  -2\frac{\left(  \theta_{j}-\theta_{p}\right)  ^{2}-\theta_{j}}{\left(
(\theta_{j}-\theta_{p})^{2}+\theta_{j}\right)  ^{2}}+\frac{3}{\left(
\theta_{j}-\theta_{p}\right)  ^{2}}\right)  \right)  }$

$+\frac{-\frac{2}{\theta_{p}^{2}}+%
{\textstyle\sum\limits_{j=1,j\neq p}^{n}}
\left(  -8\frac{\left(  \left(  \theta_{j}-\theta_{p}\right)  ^{2}-\theta
_{j}\right)  \left(  \theta_{j}-\theta_{p}\right)  }{\left(  (\theta
_{j}-\theta_{p})^{2}+\theta_{j}\right)  ^{3}}+\frac{4(\theta_{j}-\theta_{p}%
)}{\left(  (\theta_{j}-\theta_{p})^{2}+\theta_{j}\right)  ^{2}}+\frac
{6}{\left(  \theta_{j}-\theta_{p}\right)  ^{3}}\right)  }{\left(
{\textstyle\sum\limits_{j=1,j\neq p}^{n}}
\left(  \frac{2(\theta_{j}-\theta_{p})}{(\theta_{j}-\theta_{p})^{2}+\theta
_{j}}-\frac{3}{\theta_{j}-\theta_{p}}\right)  \right)  ^{2}+\left(  \frac
{2}{\theta_{p}}+%
{\textstyle\sum\limits_{j=1,j\neq p}^{n}}
\left(  -2\frac{\left(  \theta_{j}-\theta_{p}\right)  ^{2}-\theta_{j}}{\left(
(\theta_{j}-\theta_{p})^{2}+\theta_{j}\right)  ^{2}}+\frac{3}{\left(
\theta_{j}-\theta_{p}\right)  ^{2}}\right)  \right)  }.$

Now, $i\neq p,$ we have

$A_{i,1}=%
{\textstyle\prod\limits_{j=1\vee j=p,j\neq i}^{n}}
\left(  \frac{(\theta_{j}-\theta_{i})^{2}+\theta_{j}}{(\theta_{j}-\theta
_{i})^{3}}\right)  =\frac{(\theta_{p}-\theta_{i})^{2}+\theta_{p}}{(\theta
_{p}-\theta_{i})^{3}}%
{\textstyle\prod\limits_{j=1,j\neq i,p}^{n}}
\left(  \frac{(\theta_{j}-\theta_{i})^{2}+\theta_{j}}{(\theta_{j}-\theta
_{i})^{3}}\right)  ,$ then%
\begin{align}
\frac{\partial A_{i,1}}{\partial\theta_{p}}  & =%
{\textstyle\prod\limits_{j=1,j\neq i,p}^{n}}
\left(  \frac{(\theta_{j}-\theta_{i})^{2}+\theta_{j}}{(\theta_{j}-\theta
_{i})^{3}}\right)  \frac{\partial}{\partial\theta_{p}}\left(  \frac
{(\theta_{p}-\theta_{i})^{2}+\theta_{p}}{(\theta_{p}-\theta_{i})^{3}}\right)
\nonumber\\
& =\frac{-\theta_{i}-2\theta_{p}-\left(  \theta_{i}-\theta_{p}\right)  ^{2}%
}{\left(  \theta_{i}-\theta_{p}\right)  ^{4}}%
{\textstyle\prod\limits_{j=1,j\neq i,p}^{n}}
\left(  \frac{(\theta_{j}-\theta_{i})^{2}+\theta_{j}}{(\theta_{j}-\theta
_{i})^{3}}\right) \label{Eq_AI1 i diff p}%
\end{align}
and again using the same manner, we get%
\begin{align*}
\frac{\partial A_{i,2}}{\partial\theta_{p}}  & =(\left(  \frac{2(\theta
_{j}-\theta_{p})}{(\theta_{j}-\theta_{p})^{2}+\theta_{j}}-\frac{3}{\theta
_{j}-\theta_{p}}\right)  \left(  \frac{-\theta_{i}-2\theta_{p}-\left(
\theta_{i}-\theta_{p}\right)  ^{2}}{\left(  \theta_{i}-\theta_{p}\right)
^{4}}\right)  +\\
& \left(  \frac{(\theta_{j}-\theta_{p})^{2}+\theta_{j}}{(\theta_{j}-\theta
_{p})^{3}}\right)  \left(  \frac{4(\theta_{j}-\theta_{p})^{2}}{\left(
(\theta_{j}-\theta_{p})^{2}+\theta_{j}\right)  ^{2}}-\frac{2}{\left(
\theta_{j}-\theta_{p}\right)  ^{2}+\theta_{j}}-\frac{3}{\left(  \theta
_{j}-\theta_{p}\right)  ^{2}}\right) \\
& +%
{\textstyle\sum\limits_{j=1,j\neq i,p}^{n}}
\left(  \frac{2(\theta_{j}-\theta_{i})}{(\theta_{j}-\theta_{i})^{2}+\theta
_{j}}-\frac{3}{\theta_{j}-\theta_{i}}\right)  \frac{-\theta_{i}-2\theta
_{p}-\left(  \theta_{i}-\theta_{p}\right)  ^{2}}{\left(  \theta_{i}-\theta
_{p}\right)  ^{4}})\times\\
&
{\textstyle\prod\limits_{j=1,j\neq i,p}^{n}}
\left(  \frac{(\theta_{j}-\theta_{i})^{2}+\theta_{j}}{(\theta_{j}-\theta
_{i})^{3}}\right)  .
\end{align*}
and%
\begin{align*}
\frac{\partial A_{i,3}}{\partial\theta_{p}}  & =\frac{\partial}{\partial
\theta_{p}}\frac{A_{i,1}}{2}\left(  \frac{2(\theta_{j}-\theta_{p})}%
{(\theta_{j}-\theta_{p})^{2}+\theta_{j}}-\frac{3}{\theta_{j}-\theta_{p}%
}\right)  ^{2}+\\
& \frac{\partial}{\partial\theta_{p}}\frac{A_{i,1}}{2}%
{\textstyle\sum\limits_{j=1,j\neq i,p}^{n}}
\left(  \frac{2(\theta_{j}-\theta_{i})}{(\theta_{j}-\theta_{i})^{2}+\theta
_{j}}-\frac{3}{\theta_{j}-\theta_{i}}\right)  +\\
& \frac{\partial}{\partial\theta_{p}}\frac{A_{i,1}}{2}\left(  -2\frac{\left(
\theta_{j}-\theta_{p}\right)  ^{2}-\theta_{j}}{\left(  (\theta_{j}-\theta
_{p})^{2}+\theta_{j}\right)  ^{2}}+\frac{3}{\left(  \theta_{j}-\theta
_{p}\right)  ^{2}}\right)  +\\
& \frac{\partial}{\partial\theta_{p}}\frac{A_{i,1}}{2}%
{\textstyle\sum\limits_{j=1,j\neq i,p}^{n}}
\left(  -2\frac{\left(  \theta_{j}-\theta_{i}\right)  ^{2}-\theta_{j}}{\left(
(\theta_{j}-\theta_{i})^{2}+\theta_{j}\right)  ^{2}}+\frac{3}{\left(
\theta_{j}-\theta_{i}\right)  ^{2}}\right)  +\\
& A_{i,1}(2\left(  \left(  \frac{4(\theta_{j}-\theta_{p})^{2}}{\left(
(\theta_{j}-\theta_{p})^{2}+\theta_{j}\right)  ^{2}}-\frac{2}{\left(
\theta_{j}-\theta_{p}\right)  ^{2}+\theta_{j}}-\frac{3}{\left(  \theta
_{j}-\theta_{p}\right)  ^{2}}\right)  \right)  \times\\
& \left(  \left(  \frac{2(\theta_{j}-\theta_{p})}{(\theta_{j}-\theta_{p}%
)^{2}+\theta_{j}}-\frac{3}{\theta_{j}-\theta_{p}}\right)  +%
{\textstyle\sum\limits_{j=1,j\neq i,p}^{n}}
\left(  \frac{2(\theta_{j}-\theta_{i})}{(\theta_{j}-\theta_{i})^{2}+\theta
_{j}}-\frac{3}{\theta_{j}-\theta_{i}}\right)  \right)  )+\\
& A_{i,1}\left(  -8\frac{\left(  \left(  \theta_{j}-\theta_{p}\right)
^{2}-\theta_{j}\right)  \left(  \theta_{j}-\theta_{p}\right)  }{\left(
(\theta_{j}-\theta_{p})^{2}+\theta_{j}\right)  ^{3}}+\frac{4(\theta_{j}%
-\theta_{p})}{\left(  (\theta_{j}-\theta_{p})^{2}+\theta_{j}\right)  ^{2}%
}+\frac{6}{\left(  \theta_{j}-\theta_{p}\right)  ^{3}}\right)
\end{align*}
where $\frac{\partial A_{i,1}}{\partial\theta_{p}},$ and $A_{i,1}$ defined in
the above Eqs (\ref{Eq_AI1 i diff p}) and (\ref{EqAi11}).

By substituting all the above obtained equation in (\ref{eq_maxdiffcase2})
with the two different cases, we obtain an unified equation. The MLEs are
solutions of (\ref{eq_maxdiffcase2}). These equation is not solvable
analytically, but some numerical iterative methods, as Newton-Raphson method,
can be used. The solutions can be approximate numerically by using software
such as MATHEMATICA, MAPLE, and R. Here we work with MATHEMATICA.
\end{proof}

\begin{description}
\item[Application ] The data set is from \citet{9a}. The data given
arose in tests on endurance of deep groove ball bearings. The data are the
number of million revolutions before failure for each of the 23 ball bearings
in the life tests and they are:%
\begin{align*}
& 17.88,28.92,33.00,41.52,42.12,45.60,48.48,51.84,51.96,54.12,55.56,\\
& 67.80,68.64,68.64,68.88,84.12,93.12,98.64,105.12,105.84,127.92,\\
& 128.04,173.40
\end{align*}
We expect that the proposed HypoXG$(\theta_{1},\theta_{2})$ distribution will
serve as a competitive model. Using our implicit equations in Eq
(\ref{EQ_impliciteFct2}) and with the help of MATHEMATICA\ to solve these
equations, we obtain the estimation of $\theta_{1}$ and $\theta_{2}$ as
\[%
\begin{tabular}
[c]{ll}%
$\widehat{\theta}_{1}=27947.47469372068$ & $\widehat{\theta}_{2}%
=0.05407088132815127$%
\end{tabular}
\]
The estimated distribution HypoXG$(\widehat{\theta}_{1},\widehat{\theta}_{2})$
versus the estimated distribution of the Hypoexponential distribution with two
parameters i.e. Hypo$\exp(\alpha_{1},\alpha_{2})$. In \cite{4a}, Chesneau,
used the method of maximum likelihood to determine an estimation of the
parameters of the Hypoexponential distribution. The method is applied on the
survival time data and showed an estimation of $\alpha_{1}$ and $\alpha_{2}$
as
\[%
\begin{tabular}
[c]{ll}%
$\widehat{\alpha}_{1}=0.027691338927039302$ & $\widehat{\alpha}_{2}%
=0.027691606617103376$%
\end{tabular}
\]

The fitted densities are superimposed on the histogram of the data sets and
shown in the following Figure:%

\begin{figure}[h]
\centering
{%
\resizebox*{6cm}{!}{\includegraphics{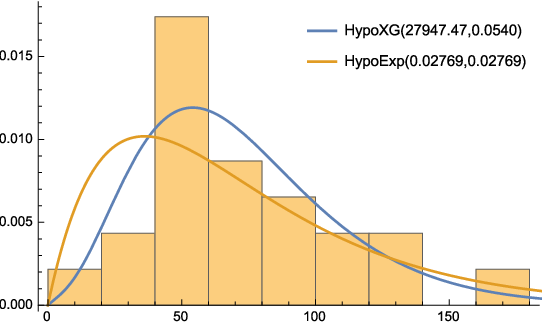}}}\hspace{5pt}
{%
\resizebox*{6cm}{!}{\includegraphics{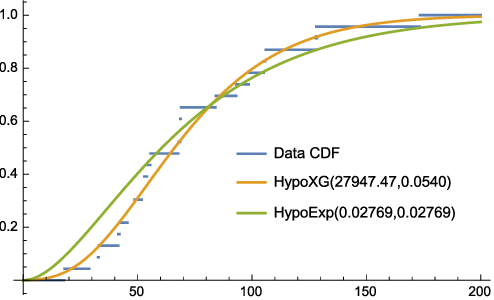}}}
\caption{Fitted pdfs and cdfs to the Survival Times for BallBearings data set.} \label{sample-figure}
\end{figure}
\end{description}


\begin{thebibliography}{}

\bibitem[Abdelkader (2003)]{1}
Abdelkader, Y.~H. (2003).
 Erlang distributed activity times in
 stochastic activity networks. \emph{Kybernetika}, \emph{39}(3), 347--358.

\bibitem[Amari \& Misra (1997)]{2}
Amari, S.~V. \& Misra, R.~B. (1997). Closed-form expressions for
distribution of sum of Exponential random variables. \emph{IEEE Transactions on
reliability}, \emph{46}(4), 519--522.

\bibitem[Bocharov et al.(2011)]{3} Bocharov, P.~P., D'Apice, C., \& Pechinkin, A.~V. (2011). \emph{Queueing theory}. Walter de Gruyter.

\bibitem[Bolch et al. (2006)]{4} Bolch, G., Greiner, S., De Meer, H., \& Trivedi, K.(2006). Queueing
Networks and Markov Chains: Modeling and Performance Evaluation with Computer
Science Applications. 2nd Edition John Wiley \& Sons. 

\bibitem[Chesneau (2018)]{4a} Chesneau, C. (2018). A new family of distributions based on the
Hypoexponential distribution with fitting reliability data. \emph{Statistica}, \emph{78}(2), 127--147.

\bibitem[Everitt (2005)]{5} Everitt, B.(2005). Encyclopedia of statistics in behavioral science. (Vol.~2). Wiley-Blackwell.

\bibitem[Kadri et al. (2015)]{7} Kadri, T., Smaili, K. \& Kadry, S. (2015). Markov modeling for
reliability analysis using Hypoexponential distribution. \emph{In Numerical Methods
for Reliability and Safety Assessment}, (pp. 599--620). Springer, Cham.

\bibitem[Khuong \& Kong (2006)]{9} Khuong, H.~V., Kong H.~Y. General expression for pdf of a sum of
independent Exponential random variables. \emph{IEEE Commun Lett}. \emph{10}(3), 159--161.


\bibitem[Lawless (2011)]{9a}Lawless J.~F. (2011). Statistical models and methods for lifetime
data. (Vol.~362). John Wiley \& Sons.

\bibitem[McLachlan \& Peel (2000)] {10} McLachlan, G.~J, \& Peel, G. (2000). Finite mixture models. New
York, NY. John Wiley \& Sons

\bibitem[Moschopoulos (1985)]{11} Moschopoulos, P.~G.(1985). The distribution of the sum of
independent Gamma random variables. \emph{Annals of the Institute of Statistical
Mathematics}, \emph{37}(3), 541--544.

\bibitem[Sen et al. (2016)]{12} Sen, S., Maiti, S.~S., \& Chandra, N. (2016). The XGamma
distribution: statistical properties and application. \emph{Journal of Modern Applied Statistical Methods}, \emph{15}(1), 774--788.

\bibitem[Smaili et al. (2014)]{21} Smaili, K, Kadri, T., Kadry, S. A Modified-Form Expressions For
The Hypoexponential Distribution, \emph{British Journal of Mathematics \&
Computer Science, }\emph{4},(3) 322--332.


\bibitem[Trivedi (1982)]{17} Trivedi, K.~S. (1982). Probability and statistics with reliability, queuing, and computer science applications (Vol.~13). Englewood Cliffs, NJ: Prentice-hall.


\end{thebibliography}
\end{document}